\documentclass[oneside]{amsart}
{}
\usepackage[dvipsnames,svgnames,x11names]{xcolor}
\usepackage{geometry}
\usepackage[nopatch=footnote]{microtype}

\usepackage{amsmath}
\usepackage{amsthm}
\usepackage{mathtools}

\usepackage{mathrsfs}
\usepackage{stmaryrd}

\usepackage{extarrows}
\usepackage[shortlabels]{enumitem}
\usepackage{tensor}
\usepackage{physics2}
  \usephysicsmodule{ab,xmat}
\usepackage{fixdif}
  \newcommand{\dd}{\d}
\usepackage{derivative}

\usepackage{graphicx}
\usepackage{subcaption}
\usepackage{comment}

\usepackage[colorlinks=true,linkcolor=blue]{hyperref}
\usepackage[hyperref=true,backend=biber,style=alphabetic,backref=true,url=false]{biblatex}

\usepackage[warnings-off={mathtools-colon,mathtools-overbracket}]{unicode-math}
\usepackage[default]{fontsetup}
\newcommand{\lap}{\increment}
\renewcommand{\emptyset}{⌀}

\DeclareMathOperator{\Ric}{Ric}

\newcommand{\ie}{\emph{i.e.}}

\newcommand{\eps}{\varepsilon}
\newcommand{\vphi}{\varphi}

\newcommand{\lrr}{\longrightarrow}

\newcommand{\op}{\operatorname}
\newcommand{\ov}{\overline}

\let\originalleft\left
\let\originalright\right
\renewcommand{\left}{\mathopen{}\mathclose\bgroup\originalleft}
\renewcommand{\right}{\aftergroup\egroup\originalright}

\theoremstyle{plain}\newtheorem{theorem}{Theorem}
\theoremstyle{definition}\newtheorem{definition}[theorem]{Definition}
\theoremstyle{definition}\newtheorem{example}[theorem]{Example}
\theoremstyle{definition}
\theoremstyle{plain}
\theoremstyle{plain}\newtheorem{corollary}[theorem]{Corollary}
\theoremstyle{plain}\newtheorem{lemma}[theorem]{Lemma}
\theoremstyle{plain}
\theoremstyle{plain}
\theoremstyle{plain}
\theoremstyle{plain}
\theoremstyle{remark}
\theoremstyle{definition}
\theoremstyle{definition}
\theoremstyle{definition}
\theoremstyle{plain}
\theoremstyle{plain}
\theoremstyle{remark}\newtheorem*{remark}{Remark}
\theoremstyle{remark}
\numberwithin{equation}{section}
\numberwithin{theorem}{section}
\numberwithin{figure}{section}

\allowdisplaybreaks{}

\title{A Frankel type theorem in Euclidean and hyperbolic spaces}
\author{Haotian Xue}

\begin{document}

\begin{abstract}
  We prove that a connected mean convex region in $\mathbb{R}^{n+1}$ with at least two boundary components
  cannot have strictly positive mean curvature. This answers a question of Gromov. We also obtain estimates
  for how quickly the mean curvature must decay at infinity, and generalize this result to hyperbolic space. 
\end{abstract}

\maketitle

\section{Introduction}
The celebrated halfspace theorem~\cite{hoffmanStrongHalfspaceTheorem1990a} by Hoffman and Meeks implies that
in \(\mathbb{R}^3\), a mean convex region with two boundary components is the slab between two parallel
planes. This fails in higher dimensions, since a catenoid of dimension greater than 3 could fit between
parallel planes. In this paper we study a variant of this problem where the boundary components are assumed
to have mean curvature with positive lower bounds. In particular, we prove

\begin{theorem}\label{gromov-conj}
  Let \(\Omega\) be a connected mean convex region in \(\mathbb{R}^{n+1}\) whose boundary has at least two
  components. Then each boundary component \(\Sigma\) is non-compact and \(\inf H_\Sigma=0\), where
  \(H_\Sigma\) is the mean curvature of \(\Sigma\) with respect to inward normals.
\end{theorem}
This resolves a conjecture by Gromov (\cite{gromovMeanCurvatureLight2020}, conjecture 1). Moreover, it can be
viewed as an analog of the Frankel property, \ie\ there do not exist two disjoint closed minimal
hypersurfaces in a positive curvature space~\cite{frankelManifoldsPositiveCurvature1961}. Here we trade the
positivity ofbackground curvature with positivity of boundary curvature. 
/ass
\begin{remark}
We note that Jian Ge has recently independently obtained this result~\cite{geUnboundedMeanConvex2026} using a
related method. In this paper we also prove a quantitative estimate for the curvature decay rate, as well as a
generalization to hyperbolic space. 
\end{remark}

To state more results, we introduce the following notation.
\begin{definition}\label{notation}
  Throughout this paper, \((M,g)\) refers to a \((n+1)\)-dimensional complete Riemannian manifold (usually a
  space form), and \(\Omega\subset M\) is called a \strong{regular region} or simply \strong{region} if it is
  a smooth open subset whose boundary is a smooth properly embedded hypersurface.

  Denote by \(H\) the mean curvature vector of \(\partial\Omega\), \(\nu\) the \strong{inward} unit normal.
  Let \(\Sigma_1,\Sigma_2,\ldots\) be the connected components of \(\partial\Omega\), and we write
  \(H_1,H_2,\nu_1,\nu_2\) for the mean curvature and unit normal on \(\Sigma_1\) and \(\Sigma_2\) respectively.

  When \(H\) is compared with a number, say \(H\ge c\), we mean \(H\cdot\nu\ge c\), and \(\inf_K H_{\Sigma_i}\)
  means \(\inf_{\Sigma_i\cap K}H_i\cdot\nu_i\). In particular, \(\Omega\) is called \emph{mean convex} if
  \(H\ge 0\) everywhere on \(\partial\Omega\).
\end{definition}

We have the following generalization of theorem~\ref{gromov-conj}:
\begin{theorem}\label{main-thm}
  Let \((M^{n+1},g)\) be the space form of constant curvature \(-\kappa^2\). Let \(\Omega\subset M\) be a 
  connected region with at least two boundary components \(\Sigma_1\) and \(\Sigma_2\), then \[
    \inf H_{\Sigma_1}+\inf H_{\Sigma_2}
    \le 2n\kappa\tanh\left(\frac{\kappa}{2}d(\Sigma_1,\Sigma_2)\right)\le 2n\kappa
  .\] In particular, when \(\kappa>0\), the last inequality is strict.
\end{theorem}

In fact for \(\mathbb{R}^{n+1}\), we are able to prove a quantitative estimate of the curvature decay rate:
\begin{theorem}\label{main-thm-est}
  Let \(\Omega\subset\mathbb{R}^{n+1}\) be a connected region with at least two boundary components
  \(\Sigma_1\) and \(\Sigma_2\). Assume in addition both \(\Sigma_1\) and \(\Sigma_2\) are non-compact. Let
  \(A_R=\{R/3<|x|<R\}\), then \[
    \inf_{A_R}H_{\Sigma_1}+\inf_{A_R}H_{\Sigma_2}\le 32nR^{-1}
  ,\] for all sufficiently large \(R\).
\end{theorem}
Note that in theorem~\ref{main-thm-est} we don't require any curvature condition on \(\Sigma_j\)'s. If we
assume global lower bound of \(H\) as in theorem~\ref{gromov-conj}, we may obtain a better decay estimate:
\begin{theorem}\label{main-thm-est-2}
  Let \(\Omega\subset\mathbb{R}^{n+1}\) be a connected region with at least two boundary components
  \(\Sigma_1\) and \(\Sigma_2\). Assume in addition \(\inf H_{\Sigma_1}+\inf H_{\Sigma_2}=0\), \ie\ the
  equality case in theorem~\ref{main-thm}. Then each of them is non-compact, and for any \(\eps>0\) and
  \(0<\beta<1\), let \(A_{\beta,R}=\{\beta R<|x|<R\}\), \[
    0\le\inf_{A_{\beta,R}}H_{\Sigma_1}+\inf_{A_{\beta,R}}H_{\Sigma_2}\le \eps nR^{-2}
  ,\] for all sufficiently large \(R\).
\end{theorem}

For hyperbolic spaces, a weaker decay estimate could be obtained:
\begin{theorem}\label{main-thm-est-Hn}
  Let \(\kappa>0\), \((M^{n+1},g)\) be the space form of constant curvature \(-\kappa^2\). Let \(\Omega
  \subset M\) be a connected region with at least two boundary components \(\Sigma_1\) and \(\Sigma_2\).
  Assume in addition both \(\Sigma_1\) and \(\Sigma_2\) are non-compact. Let \(A_R=\{R/3<|x|<R\}\), then \[
    \inf_{A_R}H_{\Sigma_1}+\inf_{A_R}H_{\Sigma_2}\le 2n\kappa+13n\kappa^{\frac{1}{3}}R^{-\frac{2}{3}}
  ,\] for all sufficiently large \(R\).
\end{theorem}

We list some examples to justify the optimality of our theorems:
\begin{example}
  Let \(M=\mathbb{H}^2\cong\{|z|<1\}\subset\mathbb{C}\) be the Poincaré disk with constant curvature \(-1\).
  Let \(\Sigma_1\) and \(\Sigma_2\) be the circles (in \(\mathbb{C}\)) \(\{|z\pm a|^2=1+|a|^2\}\), \(\Omega\)
  be the region between them, then \[
    H_{\Sigma_j}=(1+a^2)^{-\frac{1}{2}}=\tanh\left(\frac{1}{2}d_{\mathbb{H}^2}(\Sigma_1,\Sigma_2)\right)
  .\] Hence the estimate in theorem~\ref{main-thm} is optimal.
\end{example}

\begin{example}
  Let \(M=\mathbb{R}^2\), \(\Sigma_1\) be a smooth curve that agrees with the graph \(\{y=|x|/\log |x|\}\)
  when \(|x|\) is large, and \(\Sigma_2\) be the x-axis. Then \(H_{\Sigma_1}\sim r^{-1}(\log r)^{-2}\) when
  \(|x|\to\infty\). This shows that \(R^{-1}\) in theorem~\ref{main-thm-est} is the optimal power.
\end{example}

\begin{example}
  Let \(M=\mathbb{R}^4\), \(f(t)=e^{\frac{1}{1-t}}\) for \(\frac{1}{2}\le t<1\). Let \(\Sigma_1\) be the
  surface of revolution defined by \(\{\sqrt{x_1^2+x_2^2+x_3^2}=f(x_0)\}\), add a spherical cap to the
  boundary of \(\Sigma_1\) and smoothen to make it complete and has positive mean curvature. Let
  \(\Sigma_2\) be the hyperplane \(\{x_0=1\}\). Note that on \(\Sigma_1\), \(r\sim h:=f(x_0)\), the mean
  curvature of \(\Sigma_1\) is given by \[
    H_{\Sigma_1}=\frac{2+h^2(\log h)^3(\log h-2)}{h(1+h^2(\log h)^4)^{\frac{3}{2}}}\sim r^{-2}(\log r)^{-2}
    .\] This shows that \(R^{-2}\) in theorem~\ref{main-thm-est-2} is the optimal power.
\end{example}

\subsection*{Ideas of the proof}
In dimension \(n+1\le 6\), theorem~\ref{gromov-conj} can be proved by minimizing the ``brane functional''
between two boundary components of \(\Omega\) and appealing to the resolution of the Do Carmo conjecture 
(see~\cite{chenCarmosProblemCMC2025} for case \(n+1=6\) on stable constant mean curvature surfaces).
Instead we take an alternative approach inspired by Frankel's theorem. Suppose we have \(\Omega\) with two
boundary components, on which the mean curvature points inwards and \(H\ge c>0\). The
second variation of length implies that there does not exist any minimizing geodesics between the two
components.

Inspired by Ilmanen's localized maximum principle (\cite{ilmanenStrongMaximumPrinciple1996}, 
~\cite{chodoshMeanCurvatureFlow2024} appendix D) and Gromov's \(\mu\)-bubble
method~\cite{gromovMetricInequalitiesScalar2018}, we introduce a conformal metric to ensure the existence of
minimizing geodesic with second variation still under control. The resulting second variation looks like \[
  (u(p)+u(q))c\le\int_{\gamma}F(u)\dd{s}
,\] where \(u\) is the conformal factor and \(\gamma\) is the minimizing geodesic between two boundary
components with endpoints \(p\) and \(q\), \(F(u)\) is a quantity involving derivatives of \(u\).

It turns out \(F(u)\) has a good upper bound which decays fast with respect to \(R\to\infty\) by a proper
choice of \(u\) in a sufficiently large ball of radius \(R\). The main difficulty is to control the value of
\(u\) on both sides since \(u\) will go to 0 on the boundary of the ball to ensure the existence of the
minimizing geodesic and stay constant in the very interior for the \(F(u)\) bound. To overcome this, we
observe that \(u\) does not change a lot along the minimizing geodesic by an explicit comparison with a fixed
competitor, and then the right hand side could be controlled by \(o(1)\cdot u|_{\gamma}\) when \(R\to\infty\).
This leads to a contradiction to the assumption \(c>0\).

\subsection*{Acknowledgement}
I would like to express my sincere thanks to Professor Otis Chodosh for advising, and to Kai Xu for giving
many useful suggestions to improve the paper.

\section{Properties of conformal metric}
In this section, suppose \(u\) is a smooth non-negative compactly supported function on \(M\). Let
\(\tilde{g}=u^{-2}g\). Let \(\{e_0,\ldots,e_n\}\) be any (pointwise) orthonormal basis of \(g\), 
then \(\{\tilde{e}_i\}=\{u e_i\}\) is orthonormal under \(\tilde{g}\).

It is well-known that the Levi-Civita connection of \(\tilde{g}\) is given by \[
  \widetilde{\nabla}_X Y=\nabla_X Y-u^{-1}\left((\nabla_X u)Y+(\nabla_Y u)X-g(X,Y)\nabla u\right)
,\] the sectional and Ricci curvature of \(\tilde{g}\) is given by
\begin{align*}
  &\widetilde{R}(\tilde{e}_i,\tilde{e}_j,\tilde{e}_j,\tilde{e}_i)=
  u^2 R(e_i,e_j,e_j,e_i)+u(u_{ii}+u_{jj})-|\nabla u|^2, \quad (i\neq j), \\
  &\widetilde{\Ric}(\tilde{e}_j,\tilde{e}_j)=u^2\Ric(e_j,e_j)+u\lap u+(n-1)uu_{jj}-n|\nabla u|^2
,\end{align*}
where all differentiation on the right hand side is with respect to \(g\),
\(\nabla u=g^{ij}(\nabla_j u)\partial_j\) is the gradient of \(u\).
(See e.g.~\cite{leeIntroductionRiemannianManifolds2018} proposition 7.29, theorem 7.30).

Let \(\Sigma\subset M\) be a smooth hypersurface with induced metric, \(\{\nu,e_1,\ldots,e_{n}\}\) be an
orthonormal basis of \(g\) at some point on \(\Sigma\) where \(\nu\) is a unit normal. Let
\(\tilde{\nu}=u\nu,\tilde{e}_i=u e_i\) be orthonormal under \(\tilde{g}\).

The mean curvature vector of \(\Sigma\) under \(\tilde{g}|_{\Sigma}\) is given by
\begin{align*}
  \tilde{H}&=\sum_{i=1}^n (\tilde{\nabla}_{\tilde{e}_i}\tilde{e}_i)^\perp 
  =u^2 \sum_i (\tilde{\nabla}_{e_i}e_i)^\perp \\
  &=u^2\sum_{i}(\nabla_{e_i} e_i-u^{-1}(2\nabla_{e_i} u e_i-\nabla u))^\perp \\
  &=u^2(H+n u^{-1}\nabla u)^\perp
,\end{align*}
and the scalar mean curvature is given by \[
  \tilde{H}\cdot_{\tilde{g}}\tilde{\nu}=uH\cdot\nu+n\nabla_\nu u
.\] 

Let \(\gamma\) be a geodesic in \(M\) under metric \(\tilde{g}\). Let \(T\) be the unit tangent of \(\gamma\)
under \(g\) and \(\tilde{T}=uT\) be the unit tangent under \(\tilde{g}\). The geodesic equation of \(\gamma\)
is then 
\begin{align*}
  0&=\tilde{\nabla}_{\tilde{T}}\tilde{T}
  =\nabla_{uT}(uT)-u^{-1}\left(2(\nabla_{uT}u)uT-|uT|^2\nabla u\right) \\
  &=u^2\nabla_T T-(\nabla_T u)uT+u\nabla u
.\end{align*}
Define the \(g\)-curvature and unit normal of \(\gamma\) by \(\nabla_T T:=k_g N\). If \(k_g=0\),
\(N\) could be chosen as any unit vector perpendicular to \(T\). Then \[
  0=(k_g u+\nabla_N u)uN+u\cdot(\nabla u-(\nabla_T u)T-(\nabla_N u)N)
.\] For simplicity, write \(u_T=\nabla_T u\) and \(u_N=\nabla_N u\), and we conclude that the geodesic
curvature of \(\gamma\) \[
 k_g=-u^{-1}u_N
,\] and \(\nabla u\) (as a vector field) lies in the two dimensional subspace spanned by \(\{T,N\}\).
\begin{remark}
  The claim is still true when \(k_g=0\), in which \(\nabla u\) has to align with \(T\). In this case 
  \(u_N=0\) for whatever \(N\) is chosen to be. 
\end{remark}

Finally, let \(X\) be a variation field along \(\gamma\). The second variation of length of \(\gamma\) is
given by \[
  I(X,X)=\tilde{T}(q)\cdot_{\tilde{g}}\tilde{\nabla}_X X-\tilde{T}(p)\cdot_{\tilde{g}}\tilde{\nabla}_X X 
  +\int_{\gamma}|(\tilde{\nabla}_{\tilde{T}}X)^\perp|_{\tilde{g}}^2-\tilde{R}(X,\tilde{T},\tilde{T},X)
  \dd{\tilde{s}}
,\] where \(\dd{\tilde{s}}\) is the arclength parameter under \(\tilde{g}\).
(See e.g.~\cite{leeIntroductionRiemannianManifolds2018} theorem 10.22).

\section{Main estimate}
In this section, suppose \(M^{n+1}\) is the space form of constant curvature \(-\kappa^2\), which is enough 
to prove all the theorems in this paper. Note that \(M\) is homeomorphic to \(\mathbb{R}^{n+1}\).

\subsection{Construction of minimizing geodesic}\hfill\\
The following topological lemma is true by the same method as
in~\cite{samelsonOrientabilityHypersurfacesn1969a}.
\begin{lemma}\label{surface-separation}
  Let \(M\) be a complete simply connected Riemannian manifold and \(\Sigma\subset M\) be a complete
  connected properly embedded hypersurface. Then \(\Sigma\) is orientable and \(M\) is divided by \(\Sigma\)
  into exactly 2 components.
\end{lemma}

Using the notation in~\ref{notation}, suppose \(\Omega\) is a region in \(M\) with boundary components
\(\Sigma_1\) and \(\Sigma_2\). By lemma~\ref{surface-separation}, for \(i=1,2\), \(\Sigma_i\) separates \(M\)
into two components \(A_i,B_i\), assume \(\Omega\subset B_i\). Hence the whole space \(M\) is separated by
\(\Sigma_1,\Sigma_2\) into 3 components: \(A_1,B_1\cap B_2,A_2\). We may replace \(\Omega\) by \(B_1\cap B_2\)
to assume \(\Omega\) has only two boundary components.

Suppose \(u\) is not identically zero on each of \(\Sigma_i\)'s, and \(u=O(d(x,\partial \op{supp}u))\) so
that the conformal metric \(\tilde{g}\) is complete. Then there exists a shortest geodesic \(\gamma\) from 
\(\Sigma_1\cap\op{supp}u\) to \(\Sigma_2\cap\op{supp}u\). Since \(\gamma\) has to go across \(\Omega\), if it
is not fully in \(\ov{\Omega}\), we could remove the pieces outside \(\Omega\) to get a shorter one. Thus
\(\gamma\subset\ov{\Omega}\). Also, by the shortness \(\gamma\) is perpendicular to \(\Sigma_j\)'s at
the endpoints.

\subsection{Second variation of minimizing geodesic}\hfill\\
Take \(X=\vphi(s)u^{-1/2}\tilde{e}_j\), where \(\{\tilde{e}_j\}\) is a set of \(\tilde{g}\)-orthonormal parallel
normal fields along \(\gamma\), \(\vphi(0)=\vphi(L(\gamma))=1\). Taking trace of the bilinear form 
\(I(X,X)\), note that \(T(p)=\nu_1,T(q)=-\nu_2\), we have 
\begin{align*}
  0\le\op{tr}I=&-u^{-1}\tilde{T}(p)\cdot_{\tilde{g}}\tilde{H}_1+u^{-1}\tilde{T}(q)\cdot_{\tilde{g}}\tilde{H}_2
  +\sum_{i=1}^{n}\int_{\gamma}|\tilde{\nabla}_{\tilde{T}}(\vphi u^{-\frac{1}{2}}\tilde{e}_i)|^2
  -|\vphi|^2 u^{-1}\tilde{R}(\tilde{e}_i,\tilde{T},\tilde{T},\tilde{e}_i)\dd{\tilde{s}} \\
  =&-T(p)\cdot(H_1+nu^{-1}\nabla u)+T(q)\cdot (H_2+nu^{-1}\nabla u) \\
  &+\int_{\gamma}nu^2|\nabla_T(\vphi u^{-1/2})|^2-|\vphi|^2u^{-1}
  \widetilde{\Ric}(\tilde{T},\tilde{T})\dd{\tilde{s}} \\
  =&-H_1\cdot\nu_1(p)-H_2\cdot\nu_2(q)+n((\nabla_T\log u)(q)-(\nabla_T\log u)(p)) \\
  &+\int_{\gamma}u^2\left(n\left(\frac{|\vphi'|^2}{u}-\frac{\vphi\vphi'u_T}{u^2}
  +\frac{|\vphi|^2|u_T|^2}{4u^3}\right)-u^{-1}\Ric(T,T)|\vphi|^2\right)\dd{\tilde{s}} \\
  &+\int_{\gamma}|\vphi|^2\left(n\frac{|\nabla u|^2}{u}-\lap u-(n-1)u_{TT}\right)\dd{\tilde{s}}
.\end{align*}
Let \(f=\frac{1}{2}(1+\vphi^2)\),
\begin{align*}
  n((\nabla_T\log u)(q)-(\nabla_T\log u)(p))&=n(f(q)u(q)^{-1}\nabla_T u(q)-f(p)u(p)^{-1}\nabla_T u(p)) \\
  &=n\int_{\gamma}\nabla_T (fu^{-1}\nabla_T u) \dd{s} \\
  &=n\int_{\gamma}fu^{-1}\nabla_T (\nabla u\cdot T)-fu^{-2}|u_T|^2+f'u^{-1}u_T\dd{s} \\
  &=n\int_{\gamma}f\cdot(\nabla^2_{T,T}u+\nabla u\cdot(k_g N))-fu^{-1}|u_T|^2+f'u_T\dd{\tilde{s}} \\
  &=\int_{\gamma}\frac{1}{2}(1+|\vphi|^2)(nu_{TT}-nu^{-1}|\nabla u|^2)+n\vphi\vphi'u_T\dd{\tilde{s}}
.\end{align*}
So we have 
\begin{equation}\label{traced-variation}
\begin{split}
  0\le &-H_1\cdot\nu_1(p)-H_2\cdot\nu_2(q)+\int_{\gamma}u(n|\vphi'|^2-\Ric(T,T)|\vphi|^2)\dd{\tilde{s}} \\
  &-\int_{\gamma}\frac{1}{2}(1-|\vphi|^2)\left(n\frac{|\nabla u|^2}{u}-nu_{TT}\right)\dd{\tilde{s}}
  +\int_{\gamma}|\vphi|^2\left(n\frac{|u_T|^2}{4u}-(\lap u-u_{TT})\right)\dd{\tilde{s}}
\end{split}\end{equation}

Now suppose \(u=u(r)\) where \(r=d(x,0)\) is the distance function to the origin. We will choose \(u\) to be
a smooth function of \(r^2\) so that it is smooth. Note that \(\nabla u\) is collinear with \(\nabla r\),
write \[
  J_1=n\frac{|\nabla u|^2}{u}-nu_{TT}, \quad J_2=n\frac{|u_T|^2}{4u}-(\lap u-u_{TT})
,\] we have 
\begin{equation}\label{substitute-r}
\begin{split}
  J_1&=n\frac{u'(r)^2}{u(r)}-nu''(r)|r_T|^2-nu'(r)r_{TT} \\
  J_2&=n\frac{u'(r)^2}{4u(r)}|r_T|^2-u'(r)\lap r-u''(r)|r_N|^2+u'(r)r_{TT}
.\end{split} \end{equation}

\subsection{The estimate}\hfill\\
By scaling, the case \(\op{Sec}<0\) can be reduced to the case \(\op{Sec}=-1\). We consider only the standard
Euclidean and hyperbolic space in the proof. Suppose \(M=\mathbb{R}^{n+1}\), \(g\) is the Euclidean metric or
\(M=\mathbb{H}^{n+1}\) and \(g\) is the hyperbolic metric of constant curvature \(-1\). We have:

\begin{theorem}[Main estimate]\label{main-est}
  Let \(u(r)=(1-r^2R^{-2})^2\), \(\gamma\) be the shortest geodesic between \(\Sigma_1\) and \(\Sigma_2\)
  (depends on \(R\)). Suppose there is a smooth curve \(\gamma_0\) connecting \(\Sigma_1\) and \(\Sigma_2\) in
  \(\Omega\), passing through the origin. Let \(L,\tilde{L},L_0,\tilde{L}_0\) be the length of \(\gamma,
  \gamma_0\) under \(g,\tilde{g}\) respectively, suppose \(H_{\Sigma_j}\ge c_j\) on \(B(0,R)\) and \(L_0\le
  R/4\), then
  \begin{itemize}
  \item \(M=\mathbb{R}^{n+1}\), \[
    c_1+c_2\le 5nL_0R^{-2}
  .\] 
  \item \(M=\mathbb{H}^{n+1}\), let \(\alpha=\tanh\left((1+\frac{5}{9}L_0R^{-1})\frac{L_0}{2}\right)\), \[
    c_1+c_2-2n\alpha\le 5nL_0R^{-2}+8nL_0^{\frac{1}{2}}R^{-1}
  .\] 
  \end{itemize}
\end{theorem}

By the minimality of \(\gamma\) and \(u\le 1\) we get \[
  L(\gamma)=\int_{\gamma}u\dd{\tilde{s}}\le\int_{\gamma}1\dd{\tilde{s}}=\tilde{L}(\gamma)
  \le\tilde{L}(\gamma_0)=\int_{\gamma_0}\frac{\dd{s}}{u}
.\] For \(R\ge 4L_0\), let \(\mu_0=L_0R^{-1}\), we have on \(\gamma_0\): \[
  \frac{1}{u}=(1-r^2R^{-2})^{-2}\le (1-L_0^2R^{-2})^{-2}=(1-\mu_0^2)^{-2}<1+\frac{5}{9}\mu_0
,\] for \(0<\mu_0\le\frac{1}{4}\). We find that 
\begin{equation}\label{L-est}
  L\le\tilde{L}<(1+\frac{5}{9}\mu_0)L_0<\frac{7}{6}L_0
.\end{equation}

\begin{lemma}\label{shortness-lemma}
  For any \(0\le t\le L(\gamma)\), \[
    \left|\frac{u(\gamma(t))}{u(p)}-1\right|\le\frac{5}{2}\mu_0
  .\] 
\end{lemma}
\begin{proof}[Proof of lemma]
  Note that 
  \begin{align*}
    &|\nabla u|=|u'(r)\nabla r|=4rR^{-2}(1-r^2R^{-2})\le\frac{8\sqrt{3}}{9}R^{-1}, \\
    &|\nabla_{\tilde{T}}(\log u)|=|u\nabla_T (\log u)|=|\nabla_T u|\le|\nabla u|\le\frac{8\sqrt{3}}{9}R^{-1} 
  .\end{align*}
  We have \[
    \left|\log\frac{u(\gamma(t))}{u(p)}\right|
    =\left|\int_{0}^{\tilde{t}}\nabla_{\tilde{T}}(\log u)\dd{\tilde{s}}\right|
    \le \int_{\gamma}|\nabla_{\tilde{T}}(\log u)|\dd{\tilde{s}}<\frac{8\sqrt{3}}{9}R^{-1}\cdot\frac{7}{6}L_0
    <\frac{9}{5}\mu_0
  .\] Since \(|e^x-1|\le\frac{4}{3}|x|\) for \(|x|<\frac{1}{2}\), \(|u(\gamma(t))/u(p)-1|
  \le\frac{4}{3}\cdot\frac{9}{5}\mu_0<\frac{5}{2}\mu_0\) as desired.
\end{proof}

Let \(\lambda(r)=r\) or \(\sinh r\) for \(\mathbb{R}^{n+1}\) and \(\mathbb{H}^{n+1}\) respectively, we have \[
|\nabla r|=1, \quad \lap r=\frac{n\lambda'(r)}{\lambda(r)},
\quad \text{and } r_{TT}=(1-|r_{T}|^2)\frac{\lambda'(r)}{\lambda(r)}=|r_N|^2\frac{\lambda'(r)}{\lambda(r)}
.\] Then (\ref{substitute-r}) becomes 
\begin{equation}
\begin{split}
  J_1&=n\left(\frac{(u')^2}{u}-\frac{u'\lambda'}{\lambda}\right)
  -n\left(\frac{u'}{\lambda}\right)'\lambda|r_T|^2 \\
  J_2&=n\left(\frac{(u')^2}{4u}|r_T|^2-\frac{u'\lambda'}{\lambda}\right)
  -\left(\frac{u'}{\lambda}\right)'\lambda|r_N|^2
.\end{split} \end{equation}
\begin{lemma}\label{crucial-term-est}
  Let \(u(r)=(1-r^2R^{-2})^2\), \(0\le r\le R\), then:
  \begin{itemize}
  \item When \(\lambda(r)=r\), \ie\ for \(\mathbb{R}^{n+1}\), \[
      J_2\le 4nR^{-2}
    .\]
  \item When \(\lambda(r)=\sinh r\), \ie\ for \(\mathbb{H}^{n+1}\), \[
      J_1\ge -8nR^{-2},\quad J_2\le 4nR^{-2}-nu'
    .\] 
  \end{itemize}
\end{lemma}
\begin{proof}
  First we estimate \(J_2\). Note that \(u'(r)\le 0\), and we have \[
    \left(\frac{u'}{r}\right)'=8rR^{-4}\ge 0, \quad \left(\frac{u'}{\sinh r}\right)'
    =-\frac{u'}{\sinh r}\left(\frac{2r}{R^2-r^2}+\coth r-\frac{1}{r}\right)\ge 0
  .\] And
  \begin{align*}
    \frac{(u')^2}{4u}-\frac{u'}{r}=4R^{-2}, \qquad
    \frac{(u')^2}{u}-\frac{u'}{r}=4R^{-4}(3r^2+R^2)\le 16R^{-2}, \\
    \frac{(u')^2}{4u}-\frac{\cosh r}{\sinh r}u'
    =\frac{(u')^2}{4u}-\frac{u'}{r}-\left(\coth r-\frac{1}{r}\right)u'\le 4R^{-2}-u'
  .\end{align*}
  We have used the fact that \[
    0<\coth r-\frac{1}{r}<1,\quad\forall\,r>0
  .\] Combine these estimates and \(|r_T|,|r_N|\le 1\), we obtain the estimate for \(J_2\).

  Next, for \(\lambda(r)=\sinh r\) exclusively,
  \begin{align*}
    J_1&=n\left(\frac{(u')^2}{u}-\frac{u'}{r}-\left(\coth r-\frac{1}{r}\right)u'
    +\left(\frac{2r}{R^2-r^2}+\coth r-\frac{1}{r}\right)u'|r_T|^2\right) \\
  &=n\left(\frac{(u')^2}{u}-\frac{u'}{r}
    -\left(\coth r-\frac{1}{r}\right)u'|r_N|^2-8r^2R^{-4}|r_T|^2\right) \\
    &\ge -8nR^{-2}
  .\end{align*}
\end{proof}

Now we prove the main estimate for \(\mathbb{R}^{n+1}\) and \(\mathbb{H}^{n+1}\) separately:

\begin{proof}[Proof of main estimate for \(\mathbb{R}^{n+1}\)]
  We have \(\Ric=0\). Let \(\vphi=1\), apply lemma~\ref{crucial-term-est} in (\ref{traced-variation}),
  we get \[
    0\le -H_1\cdot\nu_1(p)-H_2\cdot\nu_2(q)+4nR^{-2}\tilde{L}(\gamma)
  .\] By assumption \(H_j\cdot\nu_j\ge c_j\) in \(B(0,R)\), apply (\ref{L-est})
  we have
  \begin{align*}
    c_1+c_2\le 4nR^{-2}\tilde{L}(\gamma)\le 5nR^{-2}L_0
  .\end{align*}
\end{proof}

\begin{proof}[Proof of main estimate for \(\mathbb{H}^{n+1}\)]
  We have \(\Ric=-n\). Apply lemma~\ref{crucial-term-est} in (\ref{traced-variation}) we get
  \begin{align*}
    0\le &-H_1\cdot\nu_1(p)-H_2\cdot\nu_2(q)+n\int_{\gamma}((\vphi')^2+\vphi^2)\dd{s} \\
    &+\int_{\gamma}\frac{1-|\vphi|^2}{2}\cdot 8nR^{-2}\dd{\tilde{s}}
    +\int_{\gamma}|\vphi|^2(4nR^{-2}-nu')\dd{\tilde{s}} \\
    =&-H_1\cdot\nu_1(p)-H_2\cdot\nu_2(q)+n\int_{\gamma}((\vphi')^2+\vphi^2)\dd{s}
    +4nR^{-2}L(\gamma)+4nR^{-1}\int_{\gamma}|\vphi|^2\sqrt{u}\dd{\tilde{s}}
  .\end{align*}
  Where we used the fact that \[
    -u'(r)=4rR^{-2}u^{\frac{1}{2}}\le 4R^{-1}u^{\frac{1}{2}}
  .\] Let \[
    \vphi(s)=\frac{e^{s}+e^{L-s}}{1+e^{L}}\le 1
  .\] Then \[
    n\int_{\gamma}((\vphi')^2+\vphi^2)\dd{s}=2n\tanh\frac{L}{2}
  .\] By (\ref{L-est}), \(L<(1+\frac{5}{9}\mu_0)L_0\), by assumption \(H_j\cdot\nu_j\ge c_j\) in \(B(0,R)\),
  we have
  \begin{align*}
    0&\le-c_1-c_2+2n\tanh\frac{\tilde{L}}{2}+4nR^{-2}\tilde{L}(\gamma)
    +4nR^{-1}\int_{\gamma}|\vphi|^2\sqrt{u}\dd{\tilde{s}} \\
    &\le-c_1-c_2+2n\alpha+4nR^{-2}\tilde{L}(\gamma)+4nR^{-1}\int_{\gamma}|\vphi|^2\sqrt{u}\dd{\tilde{s}} \\
  \end{align*}
  By lemma~\ref{shortness-lemma}, \(|u(\cdot)-u(p)|\le\frac{5}{2}\mu_0 u(p)\) on \(\gamma\), we have \[
    \int_{\gamma}|\vphi|^2\sqrt{u}\dd{\tilde{s}}
    \le\Big((1+\frac{5}{2}\mu_0)u(p)\Big)^{\frac{1}{2}}\int_{\gamma}\dd{\tilde{s}}
    \le\Big((1+\frac{5}{2}\mu_0)u(p)\Big)^{\frac{1}{2}}\tilde{L}(\gamma)
  .\] On the other hand, \(\int_0^L |\vphi|^2\dd{s}\le\frac{6}{5}\) for any \(L>0\), \[
    \int_{\gamma}|\vphi|^2\sqrt{u}\dd{\tilde{s}}
    \le\Big((1-\frac{5}{2}\mu_0)u(p)\Big)^{-\frac{1}{2}}\int_{\gamma}|\vphi|^2\dd{s}
    \le\frac{6}{5}\Big((1-\frac{5}{2}\mu_0)u(p)\Big)^{-\frac{1}{2}}
  .\] Thus \[
    \int_{\gamma}|\vphi|^2\sqrt{u}\dd{\tilde{s}}\le\left(\frac{1+\frac{5}{2}\mu_0}
    {1-\frac{5}{2}\mu_0}\right)^{\frac{1}{4}}\sqrt{\frac{6}{5}\tilde{L}(\gamma)}
  .\] And we obtain
  \begin{align*}
    0\le &-c_1-c_2+2n\alpha+4nR^{-2}\tilde{L}(\gamma)+4nR^{-1}\left(\frac{1+\frac{5}{2}\mu_0}
    {1-\frac{5}{2}\mu_0}\right)^{\frac{1}{4}}\sqrt{\frac{6}{5}\tilde{L}(\gamma)} \\
    \le &-c_1-c_2+2n\alpha+5nL_0R^{-2}+8nL_0^{\frac{1}{2}}R^{-1}
  .\end{align*}
  Hence \[
    c_1+c_2-2n\alpha\le 5nL_0R^{-2}+8nL_0^{\frac{1}{2}}R^{-1}
  .\] 
\end{proof}

\section{Proof of the theorems}
In this section, all the big \(O\)'s are stated with respect to \(R\to\infty\).
\begin{proof}[Proof of theorem~\ref{main-thm}]
  Fix a \(\gamma_0\) in \(\Omega\) that connects \(\Sigma_1\) and \(\Sigma_2\) such that \(L_0\le d(\Sigma_1,
  \Sigma_2)+\eps\). Without loss of generality we could choose the origin to be on \(\gamma_0\). Let \(R\to
  \infty\), apply the main estimate~\ref{main-est} to \(B(0,R)\). Take \(c_j=\inf H_{\Sigma_j}\), then
  \(c_j=O(1)\). Since \(\gamma_0\) is fixed, \(L_0=O(1)\). We get:
  
  \noindent\bullet{}
  For \(\mathbb{R}^{n+1}\), when \(R\to\infty\), \[
    c_1+c_2\le O(R^{-2})\to 0
  .\] Hence \(c_1+c_2\le 0\).

  \noindent\bullet{}
  For \(\mathbb{H}^{n+1}\), note \(\alpha=\tanh\left((1+\frac{5}{9}L_0R^{-1})\frac{L_0}{2}\right)<1\),
  when \(R\to\infty\), \[
    c_1+c_2\le 2n\alpha+O(R^{-1})\to 2n\tanh\frac{L_0}{2}
    \le 2n\tanh\left(\frac{1}{2}d(\Sigma_1,\Sigma_2)+\frac{\eps}{2}\right)
  .\] 
  Since \(\eps\) is arbitrary, the theorem is proved.
\end{proof}

\begin{proof}[Proof of theorem~\ref{gromov-conj}]
  \(\inf H_\Sigma=0\) is a direct result by applying theorem~\ref{main-thm} to \(\mathbb{R}^{n+1}\) with
  \(\kappa=0\) and assume \(H_\Sigma\ge 0\) for any boundary component \(\Sigma\). Finally, \(\Sigma\) is
  non-compact by the corollary of the following lemma.
\end{proof}

\begin{lemma}\label{distance-lemma}
  Let \(\Omega\subset\mathbb{R}^{n+1}\) be a connected region with boundary components \(\Sigma_1\) and
  \(\Sigma_2\) such that \(\inf H_{\Sigma_1}+\inf H_{\Sigma_2}=0\). Suppose \(\Sigma_j\cap B(0,R)\neq
  \emptyset\), \(j=1,2\), then the minimum of \(d(\Sigma_1\cap\ov{B(0,R)},\Sigma_2\cap\ov{B(0,R))}\) is
  achieved on the boundary, \ie\ one of the least distance pair is on \(\partial B(0,R)\).
\end{lemma}

\begin{proof}
  Suppose \(x_0\in\Sigma_1,y_0\in\Sigma_2,x_0,y_0\in B(0,R)\) achieve the distance. The straight line
  \(\ov{x_0y_0}\) must lie in \(\Omega\) by minimality. Consider \[
    \Sigma_1'=\Sigma_1\cap\ov{B(0,R)}+(y_0-x_0)
  .\] By the minimality of distance, \(\Sigma_1'\) lies fully on one side of \(\Sigma_2\cap \ov{B(0,R)}\)
  and is tangent to it at \(y_0\). By the curvature condition and maximum principle, they must be identical on the whole component. Hence the distance from the boundaries is at most \(|x_0-y_0|\).
\end{proof}

\begin{corollary}
  Let \(\Omega\subset\mathbb{R}^{n+1}\) be a connected region with boundary components \(\Sigma_1\) and
  \(\Sigma_2\) such that \(\inf H_{\Sigma_1}+\inf H_{\Sigma_2}=0\). Then both \(\Sigma_1\) and
  \(\Sigma_2\) are non-compact.
\end{corollary}
\begin{proof}
  Suppose \(\Sigma_1\) is compact, pick any fixed \(x_1\in\Sigma_1,y_1\in\Sigma_2\). Choose \(R\) large, such
  that \(\Sigma_1\subset B(0,R/100)\), \(x_1,y_1\in B(0,R)\), and \(R>100|x_1-y_1|\). Then
  \(d(\Sigma_1\cap\ov{B(0,R)}),\Sigma_2\cap\ov{B(0,R)})\) cannot be achieved on \(\partial B(0,R)\).
  This ia a contradiction to the lemma.
\end{proof}

\begin{lemma}\label{ball-lemma}
  Suppose \(M=\mathbb{R}^{n+1}\) or \(\mathbb{H}^{n+1}\), \(\Sigma\subset M\) is a smooth embedded
  hypersurface, geodesic ball \(B(x,R)\) lies fully on one side of \(\Sigma\) and touches \(\Sigma\) at a
  point \(y\). Let \(\nu\) be the inward unit normal of \(B(x,R)\) at \(y\), then
  \(H_\Sigma (y)\cdot\nu\le H_{\partial B(x,R)}(y)\cdot\nu\), where \[
    H_{\partial B(x,R)}\cdot\nu=\frac{n}{R} \quad\text{or}\quad n+\frac{2n}{e^{2R}-1}\quad\text{respectively}
  .\] 
\end{lemma}

\begin{proof}[Proof of theorem~\ref{main-thm-est}]
  \(M=\mathbb{R}^{n+1}\), let \(c_j=\inf_{A_R}H_j\), suppose for contradiction that \(c_1+c_2\ge 32nR^{-1}\).
  Without loss of generality we can assume \(c_2\ge 16nR^{-1}\). Let \(x_0 \in \Sigma_1\) such that
  \(|x_0|=2R/3\). Since \(\Sigma_1\) is non-compact and connected, \(x_0\) exists for any sufficiently
  large \(R\).

  If \(d(x_0,\Sigma_2)>R/3\), consider the ball \(B(x_0,2nc_2^{-1})\). Since \(2nc_2^{-1}\le R/8<R/3\),
  it lies inside \(A_R\) and does not touch \(\Sigma_2\). There also exists \(y_0 \in \Sigma_2\) with
  \(|y_0|=2R/3\), rotate the ball \(B(x_0,2nc_2^{-1})\) around the origin towards \(y_0\). Apply
  lemma~\ref{ball-lemma} when it touches \(\Sigma_2\) for the first time, we get \[
    \frac{c_2}{2}=H_{\partial B(x_0,2nc_2^{-1})}\ge H_{\Sigma_2}(\text{touching point})\ge c_2
  .\] This is a contradiction. Hence \(d(x_0,\Sigma_2)\le R/3\). Apply lemma~\ref{ball-lemma} again to
  \(B(x_0,d(x_0,\Sigma_2))\) and \(\Sigma_2\), we have \[
    \frac{n}{d(x_0,\Sigma_2)}=H_{\partial B(x_0,d(x_0,\Sigma_2))}
    \ge H_{\Sigma_2}(\text{touching point})\ge c_2
  .\] Hence \(d(x_0,\Sigma_2)\le nc_2^{-1}\).

  Let \(\gamma_0\) be the geodesic from \(x_0\) to \(q_0 \in \Sigma_2\) that realizes \(d(x_0,\Sigma_2)\).
  Cut the initial part of \(\gamma_0\) to make it in \(\Omega\) if needed. We have \(L(\gamma_0)
  \le nc_2^{-1}\).  Since \[
    L(\gamma_0)+R/4\le nc_2^{-1}+R/4<R/3=d(x_0,\partial A_R)
  ,\] we see \(B(q_0,R/4)\subset A_R\). Also we have \(L_0 (R/4)^{-1}\le\frac{1}{4}\). Move the origin 
  to \(q_0\) and apply the main estimate~\ref{main-est} to \(B(q_0,R/4)\) and \(\gamma_0\), we get \[
    32nR^{-1}\le c_1+c_2\le 80nL_0R^{-2}\le 80n^2c_2^{-1}R^{-2}\le 5nR^{-1}
  .\] This is a contradiction.
\end{proof}

\begin{proof}[Proof of theorem~\ref{main-thm-est-2}]
  First we assume \(d(\Sigma_1,\Sigma_2)=0\). Let \(c_j=\inf_{A_{\beta,R}}H_{\Sigma_j}\). We have \[
    d_R:=d(\Sigma_1\cap\ov{B(0,R)},\Sigma_2\cap\ov{B(0,R)})
  \] is decreasing in \(R\), and its limit when \(R\to\infty\) is \(d(\Sigma_1,\Sigma_2)\).
  By lemma~\ref{distance-lemma}, let \(x_0\in\Sigma_1,y_0\in\Sigma_2\) achieve \(d_{(1-\delta)R}\) where
  \(\delta=(1-\beta)/2\), wlog assume \(|x_0|=(1-\delta)R\). We see \(d(x_0,\Sigma_2)<d_{(1-\delta)R}\).
  Let \(y_0\in\Sigma_2\) achieves \(d(x_0,\Sigma_2)\). Note that \(d_{(1-\delta)R}\) is small when \(R\) is
  large. Apply main estimate~\ref{main-est} to \(B(y_0,\delta R-d_{(1-\delta)R})\), we see \[
    (c_1+c_2)R^2\le C(\beta)n d_{(1-\delta)R}\lrr 0 \quad \text{as }R\to\infty
  .\] This proves the theorem.

  When \(\mathcal{d}:=d(\Sigma_1,\Sigma_2)>0\), we can pick a sequence \(\{x_k\}_{k>0}\subset\Sigma_1\) and
  \(\{y_k\}_{k>0}\subset\Sigma_2\) such that \[
    \min\{|x_k|,|y_k|\}\to\infty,\quad d(x_k,y_k)\to\mathcal{d}
  .\] Up to rotation and subsequence we may assume that \(x_k-y_k\) converges to \(\mathcal{d}e_0\).
  Let \(\Sigma_2'=\Sigma_2+\mathcal{d}e_0\). Apply previous argument to \(\Sigma_1\) and \(\Sigma_2'\),
  we see for any \(0<\beta<1\), \[
    \lim_{R\to\infty}(\inf_{A_{\beta,R}}H_{\Sigma_1}+\inf_{A_{\beta,R}}H_{\Sigma_2'})R^2=0
  .\] Since \(\Sigma_2\cap A_{\beta,R}\supset(\Sigma_2'\cap A_{\frac{1+\delta}{1-\delta}\beta,(1-\delta)R})-\mathcal{d}e_0\) 
  for any small \(\delta>0\) and \(R\) sufficiently large. Hence \[
    \lim_{R\to\infty}(\inf_{A_{\beta,R}}H_{\Sigma_1}+\inf_{A_{\beta,R}}H_{\Sigma_2})R^2
    \le\lim_{R\to\infty}(\inf_{A_{\frac{1+\delta}{1-\delta}\beta,(1-\delta)R}}H_{\Sigma_1}
    +\inf_{A_{\frac{1+\delta}{1-\delta}\beta,(1-\delta)R}}H_{\Sigma_2'})R^2=0
  .\] 
\end{proof}

\begin{proof}[Proof of theorem~\ref{main-thm-est-Hn}]
  Following the same argument in the proof of theorem~\ref{main-thm-est}, let \(c_j=\inf_{A_R}H_j\), suppose
  in the contrary \(c_1+c_2\ge 2n+2nCR^{-\frac{2}{3}}\), and without loss of generality assume
  \(c_2\ge n+nCR^{-\frac{2}{3}}\).

  Let \(\delta_2=(c_2-n)/n\ge CR^{-\frac{2}{3}}\). For sufficiently large \(R\), pick \(x_0\in\Sigma_1\) such
  that \(|x_0|=2R/3\). If \(d(x_0,\Sigma_2)>R/3\), consider the ball \(B(x_0,2\delta_2^{-1})\). Since
  \(2\delta_2^{-1}\le 2R^{\frac{2}{3}}/C<R\), it lies inside \(A_R\) and does not touch \(\Sigma_2\) for
  sufficiently large \(R\). There also exists \(y_0\in\Sigma_2\) with \(|y_0|=2R/3\), rotate the ball
  \(B(x_0,2\delta_2^{-1})\) around the origin towards \(y_0\). Apply lemma~\ref{ball-lemma} when it touches
  \(\Sigma_2\) for the first time, we get \[
    n+\frac{n\delta_2}{2}\ge n+\frac{2n}{e^{4\delta_2^{-1}}-1}=H_{\partial B(x_0,2\delta_2^{-1})}
    \ge H_{\Sigma_2}(\text{touching point})\ge c_2
  .\] This is a contradiction. Hence \(d(x_0,\Sigma_2)\le 2\delta_2^{-1}\). Apply lemma~\ref{ball-lemma}
  again to \(B(x_0,d(x_0,\Sigma_2))\) and \(\Sigma_2\), we have \[
    n+\frac{n}{d(x_0,\Sigma_2)}\ge n+\frac{2n}{e^{2d(x_0,\Sigma_2)}-1}=H_{\partial B(x_0,d(x_0,\Sigma_2))}
    \ge H_{\Sigma_2}(\text{touching point})\ge c_2
  .\] Hence \(d(x_0,\Sigma_2)\le \delta_2^{-1}\).

  Let \(\gamma_0\) be the geodesic from \(x_0\) to \(q_0 \in \Sigma_2\) that realizes \(d(x_0,\Sigma_2)\).
  Cut the initial part of \(\gamma_0\) to make it in \(\Omega\) if needed. We have \(L(\gamma_0)\le 
  \delta_2^{-1}\). Since for sufficiently large \(R\), \[
    L(\gamma_0)+R/4\le\delta_2^{-1}+R/4<R/3=d(x_0,\partial A_R)
  ,\] we see \(B(q_0,R/4)\subset A_R\). Also we have \(L_0 (R/4)^{-1}<\frac{1}{4}\). Move the origin 
  to \(q_0\) and apply the main estimate~\ref{main-est} to \(B(q_0,R/4)\) and \(\gamma_0\), we get 
  \begin{align*}
    2nCR^{-\frac{2}{3}}\le c_1+c_2-2n&\le 80nL_0R^{-2}+32nL_0^{\frac{1}{2}}R^{-1} \\
    &\le 5nR^{-1}+32nC^{-\frac{1}{2}}R^{-\frac{2}{3}} \\
    &=32nC^{-\frac{1}{2}}R^{-\frac{2}{3}}+O(R^{-1})
  .\end{align*}
  When \(2C=13\), \(2nC>32nC^{-\frac{1}{2}}\), this is a contradiction.
\end{proof}

\printbibliography{}

\end{document}